\documentclass[envcountsame, envcountsect]{amsart}

\usepackage{graphicx}
\usepackage{amsfonts}
\usepackage{amssymb}
\usepackage{amsmath}

\newtheorem{theorem}{Theorem}%[section]

\newtheorem{lemma}[theorem]{Lemma}

\newtheorem{proposition}{Proposition}%[section]

\theoremstyle{definition}

\newtheorem{remark}{Remark}

\newcommand{\bbar}{\begin{pmatrix}}
\newcommand{\ebar}{\end{pmatrix}}

\newcommand{\real}{{\mathbb R}}

\newcommand{\hh}{\mathcal{H}} 
\newcommand{\kk}{\mathcal{K}}

\newcommand{\D}{{\mathbb D}}

\newcommand{\C}{{\mathbb C}}
 
\newcommand{\rhohat}{\hat \rho}

\newcommand{\bdm}{\begin{displaymath}}
\newcommand{\edm}{\end{displaymath}}
\newcommand{\beq}{\begin{equation}}
\newcommand{\beqa}{\begin{eqnarray}}
\newcommand{\beqas}{\begin{eqnarray*}}
\newcommand{\eeq}{\end{equation}}
\newcommand{\eeqa}{\end{eqnarray}}
\newcommand{\eeqas}{\end{eqnarray*}}
\newcommand{\dd}{\textup{d}}

\begin{document}

\title[Loop Group Decompositions in Almost Split Real Forms]
{Loop Group Decompositions in Almost Split Real Forms and Applications to Soliton Theory and Geometry}   
 

\author{David Brander}
\address{Department of Mathematics\\ Faculty of Science\\Kobe University\\1-1, Rokkodai, Nada-ku, Kobe 657-8501\\ Japan}
\email{brander@math.kobe-u.ac.jp}
\thanks{Research supported by Japan Society for the Promotion of Science.}

\begin{abstract}
We prove a global Birkhoff decomposition  for almost split real forms 
of loop groups, when an underlying finite dimensional Lie group is compact. 
Among applications, this
shows that the dressing action - by the whole subgroup of loops which extend holomorphically
to the exterior disc -  on the $U$-hierarchy of the ZS-AKNS systems, on curved flats and on various
 other integrable
systems, is global for compact cases. It also implies a global infinite dimensional
Weierstrass-type representation for Lorentzian harmonic maps ($1+1$ wave maps) 
from surfaces into compact symmetric spaces.  An  ``Iwasawa-type" decomposition
of the same type of real form, with respect to a fixed point subgroup of 
an  involution of
the second kind, is also proved, and an application given.

\end{abstract}


\maketitle

\section{Introduction}
Let $G$ be a compact connected semisimple Lie group, given as the fixed point subgroup
$G^\C_\rho$, where $\rho$ is a complex antilinear involution and $G^\C$ is
a complexification of $G$. We assume $G$ is embedded as a subgroup of some matrix group.
 Let $\Lambda G^\C$ denote the complex Banach Lie group of loops
in $G^\C$, with some $H^s$-topology, $s> 1/2$. 
In the study of loop groups by Pressley and Segal \cite{pressleysegal}, 
among the many interesting facts proved are 
 two loop group decompositions which have
been very useful in applications to integrable systems. 
The \emph{Birkhoff decomposition} states 
that the set
\beq \label{birkhoff}
\mathcal{B} := \Lambda^- G^\C \, \cdot \, \Lambda ^+ G^\C
\eeq
 is open and dense  in $\left[ \Lambda G^\C \right]_e$, 
where 
$\Lambda^\pm G^\C$ denotes the subgroup of loops which are boundary values for holomorphic
maps $\D^\pm \to G^\C$,
 $\D^+$ denotes the unit disc, $\D^-$ the complement of its closure in the 
Riemann sphere, and $[X]_e$ denotes the identity component of
any group $X$. The set $\mathcal{B}$ is called the \emph{big cell}.
The \emph{Iwasawa decomposition} is
\beq \label{iwasawa}
\Lambda G^\C = \Omega G \, \cdot \, \Lambda^+ G^\C,
\eeq
 where $\Omega G \subset \Lambda G$
 is the subgroup of 
\emph{based} loops in $G$, that is loops which map 1 to the identity element.

A real form $\Lambda G^\C _{\hat \theta}$
of $\Lambda G^\C$ is defined to be the fixed point subgroup with respect to a
complex antilinear involution $\hat \theta$ of $\Lambda G^\C$. 
As part of a natural theory of affine Kac-Moody Lie groups,
finite order (complex linear or complex antilinear) automorphisms
 of loop groups have been studied and classified, based on the work of many
people - see \cite{bauschrousseau,  heintze, kacpeterson, levstein} and associated references.

After passing to an isomorphic loop group (see \cite{heintze}),
 one can assume that an antilinear involution $\hat \theta$ of $\Lambda G^\C$
has one of the following  forms:
\bdm
(\hat \theta x)(\lambda) := \theta (x(\varepsilon(\bar{\lambda})^{j})),
\hspace{1cm} \varepsilon,~j \in \{1,-1\},
\edm
where $\theta$ is an antilinear involution of $G^\C$,
 $\lambda$ is the loop parameter, and $x \in \Lambda G^\C$.
Note that, since $\lambda$ is an ${\mathbb S}^1$ parameter, one could
equivalently write $\lambda^{-1}$ instead of $\bar \lambda$, but our choice
of notation, which is commonly used, gives the right formula when one considers
holomorphic {extensions} of loops away from ${\mathbb S}^1$. For example, if $\theta$ is
just complex conjugation, and $x = \sum_i a_i \lambda^i$, and we take $\varepsilon=j=1$, then,
according to our choice,
$(\hat \theta x) (\lambda) = \sum_i \overline{a_i \bar{\lambda}^i} = \sum_i \overline{a_i} \lambda^i$.  Otherwise one 
 obtains $\sum_i \overline{a_i} \bar{\lambda}^{-i}$,
which is not holomorphic.

Complex linear involutions can be put into a similar
standard form, where $\theta$ is $\C$-linear, and $\lambda$ is not conjugated.
We consider involutions only in these standard forms. They have some redundancy
which need not concern us here.

Automorphisms come in two types, those of the \emph{first kind}, which, in 
the standard forms used in this article,
 preserve $\Lambda^+G^\C$,
and those of the \emph{second kind}, which take $\Lambda^+G^\C \to \Lambda^-G^\C$.
An involution is  of the first kind if $j=1$ and 
of the second kind if $j=-1$. 

A real form is said to be  \emph{almost split} if the involution is of the first
kind, and \emph{almost compact} if of the second kind.
The real form $\Lambda G$ is almost compact,  
as it is the fixed point subgroup of $\hat \theta$, where 
$(\hat \theta x)(\lambda) = \rho (x(\bar \lambda ^{-1}))$.
There is thus clearly no hope of obtaining an analogue of the
Birkhoff decomposition theorem for  $\Lambda G$, 
as elements of $\Lambda ^\pm G$
are just constant loops. On the other hand, the theorem does hold for
\emph{almost split} real forms
 - in fact, more generally, it is not difficult to show (see \cite{branderdorf}) that,
for any finite order automorphism  of the first kind,  $\hat \theta$
 (given in a standard form analogous to that above), the set  
$ \Lambda^- G^\C_{\hat \theta} \, \cdot \, \Lambda^+ G^\C_{\hat \theta}$
 is open and dense in $\left[ \Lambda G^\C_{\hat \theta} \right]_e$. 

\subsection{Results}
In this note we prove that 
this open dense set is actually 
the whole of $ \Lambda G^\C_{\hat \theta}$,
 for  almost split real forms when $G$ is compact:  suppose $\rhohat$ is an extension
of $\rho$ to an involution of the first kind of  $\Lambda G^\C$, given in
standard form, i.e. by one of the formulae:
\beq \label{standardform}
(\hat{\rho} x)(\lambda) := \rho (x( \varepsilon \bar \lambda)),  \hspace{1cm} \varepsilon = \pm 1.
\eeq
Set
\bdm
\hh :=  \Lambda G^\C_{\rhohat}, \hspace{1.5cm} 
\hh^{\pm} := \hh \cap \Lambda^\pm G^\C, \hspace{1cm} \hh^0 := \hh \cap G^\C = G.
\edm
Let $\Lambda^-_* G^\C$ and $\Lambda^+_* G^\C$ denote the subgroups of loops in 
$\Lambda^\pm G^\C$ which respectively
map the elements $\infty$ and $0$ to the identity element, and define 
$\hh^\pm_* := \hh \cap \Lambda^\pm_* G^\C$.
%-----------------------
\begin{theorem} \label{thm1}
 There exists a decomposition
\bdm
\hh = \hh^-_* \, \cdot \, \hh^+ .
\edm
The map $\hh^-_*  \times  \hh^+ \to \hh$, given by $(x,y) \mapsto xy$, is a real analytic diffeomorphism. The subgroup $G= \hh^0$ is a deformation retract of $\hh$. In particular, 
$\hh$ is connected.
\end{theorem}
%--------------------------------
 We derive Theorem \ref{thm1} directly,
  from the Birkhoff factorization of Pressley/Segal \cite{pressleysegal}, for the
case $G=U(n)$, and
then use the fact that any compact Lie group can be embedded in some $U(n)$.
The result does not hold if $G$ is not compact - see Remark \ref{noncompactremark} below.
There is, moreover, a generalization, Theorem \ref{thm1a} to certain subgroups of
$\hh$,  which is useful for applications, as described in 
Section \ref{iwasawaapp} below.

Let $\hat \tau: \Lambda G^\C \to \Lambda G^\C$ be an arbitrary (either $\C$-linear
or $\C$-antilinear)
involution of the \emph{second} kind
which restricts to an involution of $\hh$. Let $\hh_{\hat \tau}$ denote the fixed point
subgroup.
As a corollary of Theorem \ref{thm1} we obtain the following 
 global Iwasawa-type decomposition:
%-----------------------------
\begin{theorem} \label{thm2}
Every element $x \in \hh$ can be decomposed
\beq   \label{iwasawadecomp}
x = z_\tau y_+, \hspace{1cm} z_\tau \in \hh_{\hat \tau}, \hspace{.5cm}
  y_+ \in \hh^+.
\eeq
The element $z_\tau$ is unique up to right multiplication by an element of
 $\hh_{\hat \tau}^0 := \hh_{\hat \tau} \cap G^\C$,
and the projection $\hh \to \frac{\hh_{\hat \tau}}{\hh_{\hat \tau}^0}$,
given by $x \mapsto [z_\tau]$, the equivalence class of $z_{\tau}$ modulo this right multiplication,
is real analytic.
\end{theorem}
There is also a generalization of Theorem \ref{thm2} to subgroups, given by
Theorem \ref{thm2a}.

\subsection{Applications} \label{applications}
\subsubsection{Applications of Theorem \ref{thm1}}
The Birkhoff decomposition in almost split real forms is used
in  important techniques for producing solutions to integrable systems:
one such method is the
\emph{dressing technique}, originating in the work 
of Zakharov and Shabat \cite{zs}.
The dressing technique can easily be shown to apply in general
to a large class of integrable systems, discussed in \cite{branderdorf}.
In the present context,  this is an (at least local) action by the group $\Lambda^- G^\C_{\rhohat}$
on the space of solutions, where $G$ is not necessarily compact but $\rhohat$ is
of the first kind.
 Special cases of dressing are the classical B\"acklund
and Ribaucour transformations. 

Terng and Uhlenbeck \cite{ternguhlenbeck1998} studied the dressing action
on a class of rapidly decaying solutions for certain integrable systems
  associated to the group $G= U(n)$.  They showed that the  subgroup of \emph{rational}
 loops in $\Lambda^- G^\C_{\rhohat}$ acts globally on this class of solutions.
Dressing by rational loops is also treated more generally in almost split real forms
by Donaldson, Fox and Goertsches \cite{dfg}. 

The  decomposition Theorem \ref{thm1} strengthens the dressing result of Terng/Uhlenbeck
by showing  that the dressing 
action of  the whole subgroup $ \Lambda^- G^\C_{\rhohat}$
is well defined globally, and that this holds for any compact $G$.

There are many examples to which this applies: for example,
any of the integrable systems
described in \cite{terng2006}
which satisfy the ``$U$ reality condition" -
which is the case $\varepsilon = 1$ in (\ref{standardform}) - if $U$ is compact.
  A large geometrical
class are
\emph{curved flats} in symmetric spaces, defined by Ferus and Pedit \cite{feruspedit1996II}.
Curved flats  themselves contain, as special cases, many problems in geometry, such as
isometric immersions of space forms \cite{feruspedit1996II}, and 
 \emph{isothermic surfaces} \cite{bhjpp, cgs}, examples of which are surfaces of revolution, quadrics
  and constant mean curvature surfaces: a good
survey can be found in Burstall \cite{burstall}. 

Another important class of maps to which dressing has been 
applied successfully are \emph{harmonic maps} from  Riemann surfaces into 
compact symmetric spaces.  For Riemannian harmonic maps, the loop group is $\Lambda G$,
and dressing  is done via the standard Iwasawa splitting of Pressley and Segal. This 
allowed various interesting results to be proved 
\cite{bergveltguest, burstallpedit1995, guestohnita,  uhlenbeck1989}.
On the other hand, \emph{Lorentzian} harmonic maps (i.e. maps which are harmonic
with respect to a Lorentzian metric, also called \emph{$1+1$ wave maps}),
 from surfaces into  symmetric spaces,
 have a loop group formulation \cite{terng1997}
into an almost split real form, and dressing can be done via a Birkhoff decomposition.
Dressing of these maps by certain rational loops (simple elements)  is studied in \cite{ternguhlenbeck2004}.
The significance of Theorem \ref{thm1} is, again, that the dressing action by the whole of
$\Lambda^- G^\C_{\hat \rho}$  is global, for
compact targets, just as in the Riemannian case.

A second technique which uses the Birkhoff decomposition in an almost split real form 
is the method of Dorfmeister,
Pedit and Wu (DPW) \cite{dorfmeisterpeditwu}, when applied to Lorentzian 
harmonic maps from a surface into a symmetric space.  This technique, while
depending on a loop group factorization, differs from the dressing technique
in that it gives a way to produce \emph{all} solutions to the problem at hand
from simple data. The DPW method was shown to apply in this
way to Lorentzian harmonic maps by M Toda  \cite{ toda2005}.
Its application to pseudospherical surfaces is based on the fact that
the Gauss map of a pseudospherical surface is  harmonic with respect
to the Lorentzian metric given by its second fundamental form.
The DPW technique allows one to  construct all solutions from arbitrary pairs of curves. 
A similar result was also obtained for the sine-Gordon equation (solutions of
which correspond to pseudospherical surfaces)  by Krichever \cite{krichever}. 
 Since the group here is compact,
Theorem \ref{thm1} shows that the construction is global.  It does not (and cannot)
allow one to produce \emph{complete} pseudospherical surfaces, however, as Hilbert
proved \cite{hilbert1} that none exist. 

\subsubsection{Applications of Theorem \ref{thm2}}
The Iwasawa type decomposition Theorem \ref{thm2} is relevant to 
both the dressing action and 
the generalized DPW method, 
given in \cite{branderdorf}, for the case of isometric 
immersions of space forms (as formulated by Ferus and Pedit \cite{feruspedit1996}),
as well as to constant curvature Lagrangian submanifolds of $\C P^n$ and $\C {\mathbb H}^n$,
and, more generally, any solutions of the 3-involution loop group system studied
in \cite{reflective}.  For isometric immersions of space forms, an interesting
point here is that, even when the isometry group  is \emph{non}-compact, it turns out
that in some cases the relevant loop group also satisfies a reality condition of
the first kind for some \emph{compact} Lie group, which means that
 all solutions can be constructed globally from curved flats, and
that the dressing action is global.  This example is discussed in more
detail in Section  \ref{iwasawaapp} below.

%===================================================
\section{Proof of Theorem \ref{thm1}}
We will later be concerned with 
 automorphisms of the first kind of order $n$, 
which, for $\C$-linear and $\C$-antilinear cases respectively, we assume are in a standard form:
\beq \label{standardform2}
\begin{split}
& (\hat \theta x)(\lambda) := \theta (x( e^{\frac{2 k \pi i}{n}}\lambda)),
\hspace{1cm} k \in {\mathbb Z}, ~\hspace{.5cm} \theta ~ \textup{$\C$-linear},\\
& (\hat \theta x)(\lambda) := \theta (x( \varepsilon \bar \lambda)),
\hspace{1cm} \varepsilon =\pm 1, ~\hspace{.5cm} \theta ~ \textup{$\C$-antilinear},
\end{split}
\eeq
where $\theta$ is a finite order automorphism of $G^\C$. The antilinear case occurs
only for even values of $n$.

We first discuss the topology of some subgroups of $\Lambda^\pm G^\C$.
For any subgroup $\mathcal{K}$, define 
$\mathcal{K}^\pm = \mathcal{K} \cap \Lambda^\pm G^\C$ and 
$\mathcal{K}^\pm_* = \mathcal{K} \cap \Lambda^\pm _*G^\C$.

\begin{lemma} \label{topologylemma}
Let $\hat \theta_j$, for $j=1,...,k$, be a collection of mutually commuting 
finite order  automorphisms of the first kind of $\Lambda G^\C$, each given
as an extension of a $\C$-linear or $\C$-antilinear automorphism, $\theta_j$ of $G^\C$,
 in the form (\ref{standardform2}).  Let 
\bdm
\mathcal{K} := \Lambda G^\C_{\hat \theta_1...\hat\theta_k},
\edm
be the subgroup of elements which are fixed by all $k$ automorphisms. 
Then: 
\begin{enumerate}
\item
The subgroups $\mathcal{K}^\pm_*$ are contractible.
\item 
The constant subgroup $\mathcal{K}^0 = G^\C_{\theta_1 ... \theta_k}$
is a deformation retract of each of
 $\mathcal{K}^\pm$.
\end{enumerate}
\end{lemma}
%-----------------
\begin{proof}
We give the details for $\kk^+$, the argument being the same for the other
cases.
 Let $x$ be an 
element of $\kk^+$.  By definition, $x$ has a holomorphic extension to a map
$\hat{x}: \D^+ \to G^\C$.
Now for $\lambda \in {\mathbb S}^1$, $\hat x$ is fixed by each of $\hat \theta_j$,
that is, if $\hat \theta_j$ is $\C$-linear,
\bdm
\hat x(\lambda) \, [\theta_j(\hat x( e^{\frac{2 k \pi i}{n}}\lambda))]^{-1} -I =0,
\edm
and an analogous expression if antilinear.
Either   of these 
expressions are holomorphic
in $\lambda$ on $\D^+$, and therefore equivalent to zero on $\D^+$. 
Now consider, for $t$ in the interval $[0,1] \subset \real$,
the family of loops 
$\gamma_t: {\mathbb S}^1 \to \Lambda G^\C$
defined by 
\bdm
\gamma_t(\lambda) := \hat x(t\lambda).
\edm
This family depends continuously on $t$ because $\hat x$ is holomorphic on 
$\D^+$.
Clearly $\gamma_t \in \Lambda^+ G^\C$ for any $t \in [0,1]$. Moreover,
 we just showed that,
for the two respective cases, we have 
\bdm
\hat x(z) = \theta_j (\hat x( e^{\frac{2 k \pi i}{n}}  z )), \hspace{1.5cm}
\hat x( z) = \theta_j (\hat x( \varepsilon \bar z  )),
\edm
for any 
value $z \in \D^+$. In particular, for $\lambda \in {\mathbb S}^1$,
and real $t \in [0,1]$, we have, for the linear case:
\bdm
\gamma_t(\lambda) := \hat x(t \lambda) = 
     \theta_j (\hat x(t  e^{\frac{2 k \pi i}{n}} \lambda   ))
=: (\hat \theta_j \gamma_t) (\lambda),
\edm
and, similarly, $\gamma_t(\lambda) = (\hat \theta_j \gamma_t) (\lambda)$
for the antilinear case. 
In other words, $\gamma_t$ is fixed by all of the involutions, so
  $\gamma_t \in \kk^+$ for all $t \in [0,1]$.
  
   Finally, $\gamma_1 = x$,
$\gamma_0$ is just the constant loop  $\hat x(0) \in \kk^0$, and if 
$x \in \kk^0$ then $\gamma_t = x$ for all $t$. Hence $\Gamma: \, [0,1] \times \hh^+ \to \hh^+$,
given by $(t, x) \mapsto \gamma_t$ is a deformation retract of $\hh^+$ onto $\hh^0$.
\end{proof}

%*********
The proof of Theorem \ref{thm1} will be derived from the
Birkhoff theorem of Pressley and Segal, which states:
%-----------------------------
\begin{theorem} \label{birkhoffthm}(Birkhoff factorization theorem) \cite{pressleysegal}.
Every element $x \in \Lambda GL(n,\C)$ has a decomposition:
\beq \label{psfact} 
 x = x_- D  x_+, \hspace{1.5cm} D = \textup{diag}(\lambda^{k_1},..., \lambda^{k_n}),
\eeq
where $k_i$ are integers, $k_1 \geq ... \geq k_n$,
and $x_{\pm} \in \Lambda^\pm GL(n,\C)$.
  The set $\mathcal{B}$ 
(the big cell),  on which 
$D=I$, is open and dense in the identity component $\left[ \Lambda GL(n,\C) \right]_e$.
The multiplication map $\Lambda_*^- GL(n,\C) \times \Lambda^+ GL(n,\C) \to \mathcal{B}$
is a real analytic diffeomorphism.

The same statement holds replacing $GL(n,\C)$ with  the complexification $G^\C$
of any connected compact semisimple Lie group $G$, and replacing the middle
term $D$ with a homomorphism 
  from ${\mathbb S}^1$ into the complexification of
  a maximal torus of $G$.
\end{theorem}
%-----------------------------
Now let  $G$ and $\rho$ be as in the introduction,  suppose that
$\rhohat$ is given by (\ref{standardform}), and  
$\hh := \Lambda G^\C_{\rhohat}$.
For the subgroup $\hh$, define  the \emph{big cell} to be the set
\bdm
\mathcal{B}_{\rhohat} := \mathcal{B} \cap \hh.
\edm
  We first recall some
properties of $\hh$, which are proved in \cite{branderdorf}, but can 
be verified by the reader without much difficulty:
%************************
\begin{lemma} \cite{branderdorf}. \label{branderdorflemma}
\begin{enumerate}
\item The group $\hh$ is a closed Banach Lie subgroup of $\Lambda G^\C$.
\item If $x=x_- x_+$ is the Birkhoff splitting of an element
 $x \in \mathcal{B}_{\rhohat} \subset \Lambda G^\C$,
with $x_- \in \Lambda^-_* G^\C$ and $x_+ \in \Lambda^+ G^\C$, then 
both $x_-$ and $x_+$ are elements of $\hh$.
\item Let $x$ be an element of $\hh$, and let $D(x)$ denote the largest domain in
$\C \cup \{ \infty\}$ to which $x$ extends analytically. Then $x$ is $G$-valued for values of
$\lambda \in \real \cap D(x)$ if  $\varepsilon = 1$, and for $\lambda \in i\real \cap D(x)$
 if $\varepsilon = -1$.
\end{enumerate}
\end{lemma}

%==============================
The proof of Theorem \ref{thm1} will follow from the next proposition:
\begin{proposition} \label{prop1}
$\mathcal{B}_{\rhohat} = \hh$.
\end{proposition}
\begin{proof}
First consider the case $G= U(n)$.
Set $G^\C := GL(n,\C)$. The real form $U(n)$ is given by the involution 
$\rho x := (\bar{x}^t)^{-1}$. We describe the case $\varepsilon =1$ in 
(\ref{standardform}), the proof for  $\varepsilon =-1$ being essentially identical.
Extending $\rho$ to an involution of the loop 
group by the formula 
\beq \label{su2inv}
(\rhohat x)(\lambda) := \rho (x(\bar \lambda)) = (\overline{x(\bar \lambda)}^t)^{-1} =: (x^*)^{-1},
\eeq
it is simple to check that loops in $\Lambda G^\C_{\rhohat}$ take values in $U(n)$
when $\lambda$ is real. According to the Birkhoff factorization theorem of 
\cite{pressleysegal}, every element of $\left[ \Lambda G^\C \right]_e$ which is \emph{not} in the
big cell can be expressed in the form (\ref{psfact}),
where   at least one $k_i$ is nonzero, and $x_{\pm} \in \Lambda^\pm G^\C$.
We need to prove that such an element cannot be in the real form $\Lambda G^\C_{\rhohat}$,
in other words cannot be fixed by $\rhohat$. It then follows that every element in
$\hh$ is in fact in the big cell. 
The condition $x= \rhohat x$ is 
\bdm
x_-  \, \textup{diag} (\lambda^{k_1},..., \lambda^{k_n})\,  x_+ = (x_-^*)^{-1} 
  \textup{diag} \, (\lambda^{-k_1},..., \lambda^{-k_n}) \, (x_+^*)^{-1}.
\edm
Either $k_1 >0$ or $k_n <0$ (or both). If $k_n<0$ rearrange the equation to:
\beq \label{useeqn1}
x_-^* x_- \, \textup{diag} (\lambda^{k_1},..., \lambda^{k_n})\, = 
  \textup{diag} \, (\lambda^{-k_1},..., \lambda^{-k_n}) \, (x_+^*)^{-1} x_+^{-1}. 
\eeq
But the components of the matrix $x_-^* x_-$  are 
 power series in $\lambda^{-1}$, while those of the matrix $(x_+^*)^{-1} x_+^{-1}$ 
 are power series in $\lambda$. Since $k_n <0$, it follows that the $(n,n)$ component of
(\ref{useeqn1}) is identically zero.  
Writing the last row of $x_-^*$ as $(a_1,...,a_n)$  we thus obtain, from the $(n,n)$ component on the
left hand side, the identity
\bdm
a_1^* a_1 + ... + a_n^* a_n = 0,
\edm
where, for a scalar function $y(\lambda)$, we use
 $(y^*)(\lambda) := \overline{y(\bar \lambda)}$.
Since, for real values of $\lambda$, we have $a_i^*= \overline{a_i}$,
this implies that $a_1= a_2 = .... = a_n = 0$, at $\lambda =1$,
which is impossible, as $x_-$ is an invertible matrix for all $\lambda$ in 
the unit circle.

If $k_1>0$, we instead write
\bdm
 \textup{diag} (\lambda^{k_1},..., \lambda^{k_n})\,  x_+ x_+^* =
 x_-^{-1} (x_-^*)^{-1}  \,  \textup{diag} \, (\lambda^{-k_1},..., \lambda^{-k_n}), 
\edm
and similarly deduce that a row (this time the first) of $x_-^{-1}$ is zero.  

The general case now follows from what we have just shown:
by the Peter-Weyl theorem, the compact connected semisimple Lie group $G$ can be embedded as a 
closed subgroup of the unitary group $U(n)$ for some $n$. 
While there are many possible complexifications of $G$, they are all
isomorphic. Given the embedding in $U(n)$, and the complexification 
$U(n)^\C = GL(n,\C)$, there is a 
 natural complexification $G^\C$  as follows:
  if $\mathfrak{g}$ is the Lie algebra, set $\mathfrak{g}^\C := \mathfrak{g}+ i \mathfrak{g}$
  in $\mathfrak{gl}(n,\C)$, and
  $G^\C := \exp( \mathfrak{g}^\C)$
in $GL(n,\C)$.  It follows from the corresponding properties on the Lie algebras that
 the involution of $G^\C$ which determines $G$ is just the 
restriction to $G^\C$ of that which determines $U(n)$, and we denote both
by $\rho$. Thus $G= G^\C_\rho = G^\C \cap U(n)$.

Now $\Lambda G^\C$ is a subgroup of $\Lambda GL(n,\C)$, and,
 by assumption, $\rho$ is extended to $\Lambda G^\C$ by the formula
(\ref{standardform}). Hence this extension is also the restriction to
$\Lambda G^\C$ of the involution $\rhohat$ discusses in the $U(n)$ case.
We denote both extensions by $\rhohat$.
Thus $\Lambda G^\C_{\rhohat} =\Lambda G^\C \cap \Lambda GL(n,\C)_{\rhohat}$.

Theorem \ref{birkhoffthm} 
says that any element $x \in \Lambda G^\C$ which is not in the big cell has
  a factorization of the form (\ref{psfact}), where now $D$ is a non-trivial homomorphism 
  from ${\mathbb S}^1$ into the complexification of
  a maximal torus of $G$.
 Since any  torus of
   $G$ is contained in a maximal
  torus of $U(n)$, it follows that $D$ is a non-trivial homomorphism into
  the complexification of a maximal torus of $U(n)$.  
  Since $\Lambda G^\C_{\rhohat} \subset \Lambda GL(n,\C)_{\rhohat}$
  and
   $\Lambda ^\pm G^\C \subset \Lambda^\pm GL(n,\C)$,
   the $U(n)$ case implies that $x$ cannot be
  in $\hh$.  
\end{proof}
%*****************
 \noindent \textbf{Proof of Theorem \ref{thm1}:} $~~$
It follows from the Birkhoff factorization Theorem \ref{birkhoffthm},
from Proposition \ref{prop1}, and from the second item of Lemma 
\ref{branderdorflemma}, together with the fact that $\hh_*^-$ and $\hh^+$
are closed Banach submanifolds, that the multiplication map $\hh_*^- \times \hh^+ \to \hh$
is a real analytic diffeomorphism.  

It remains to show that $G = \hh^0$ is a deformation retract of $\hh$.
Since $\hh$ is diffeomorphic to the product $\hh_*^- \times \hh^+$,
this  follows
from Lemma \ref{topologylemma}.
\qed \\

\begin{remark}  \label{remark1}
 There is obviously an analogue of Theorem \ref{thm1},
 switching $+$ and $-$. 
\end{remark}

\begin{remark}  \label{noncompactremark}
 The decomposition of Theorem \ref{thm1}  does not hold (even if one restricts to
 the identity component) if the real form $G$ is non-compact.
 For example, taking $G^\C = SL(2,\C)$ and $G= SL(2,\real)$, the loop
  {\small{$x = \bbar \lambda & 0 \\ 0 &\lambda^{-1} \ebar$}} is not in the big cell, but does
  take values in $SL(2,\real)$ for $\lambda \in \real$ - more specifically, it is an
  element of the identity component of $\Lambda SL(2,\C)_{\rhohat}$, 
  where $(\rhohat x)(\lambda):= \overline{x(\bar \lambda)}$ defines an involution of the
  first kind. 
  
  Neither is it true, if $G^\C_{\rhohat}$ is non-compact,
   that  $\hh^0$ has the same number of connected components as $\hh$.
    The connected
  components of $\Lambda GL(n,\C)$ correspond to the winding numbers of the
  determinants of the elements.  If one considers $G=GL(2,\real)$ then the loops
  {\small{$x = \bbar \lambda^k & 0 \\ 0 &1  \ebar$}}, which have winding number $k$,
  are all elements of $\Lambda GL(2,\C)_{\rhohat}$,  
  where again $(\rhohat x)(\lambda):= \overline{x(\bar \lambda)}$ is of the first 
  kind.  Thus $\hh = \Lambda GL(2,\C)_{\rhohat}$ has infinitely many components, while
  $\hh^0 = GL(2,\real)$ has only two.
\end{remark} 

\begin{remark} 
Note that Theorem \ref{thm1} implies that the winding number, given by the 
determinant of the matrix, for any loop in 
$\Lambda GL(n,\C)_{\rhohat}$, where $\rhohat$ is given by (\ref{su2inv}),
is zero.  This follows because the theorem implies that $\Lambda GL(n,\C)_{\rhohat}$
is contained in the identity component of $\Lambda GL(n,\C)$.  
\end{remark}

\subsection{Generalization of Theorem \ref{thm1}}

 Let $G$ and $\rhohat$ be as in previous sections.
Let $\hat \theta_i$, $i = 1,...,k$ be a collection of
mutually commuting finite order 
 automorphisms of the first kind (either linear or antilinear), given in the form (\ref{standardform2}),
and all of which commute with $\rhohat$. Set
\bdm
 \widehat{\hh} := \Lambda G^\C_{\rhohat \hat \theta_1 ... \hat \theta_k},
\edm
 the  subgroup
 of $\Lambda G^\C$ consisting of elements which are fixed by all of
 $\rhohat$ and $\hat \theta_i$.

 \begin{theorem} \label{thm1a}
 There exists a decomposition
\bdm
\widehat \hh = \widehat \hh^-_* \, \cdot \, \widehat \hh^+ .
\edm
The map $\widehat \hh^-_*  \times  \widehat \hh^+ \to \hh$, given by
 $(x,y) \mapsto xy$, is a real analytic diffeomorphism. 
 The constant subgroup 
$\widehat \hh^0 = G_{\hat \theta_1 ... \hat \theta_k}$ is a deformation
retract of $\widehat \hh$.
\end{theorem}
\begin{proof}
Set $\widehat{\mathcal{B}} = \widehat{\hh} \cap \mathcal{B}$.
It is not difficult to show (see \cite{branderdorf}, Proposition 1)
 that the first two properties of Lemma
 \ref{branderdorflemma} also hold for $\widehat \hh$.  Since $\widehat{\hh}$ is
 contained in $\hh$, Proposition \ref{prop1} also applies to $\widehat{\mathcal{B}}$.
 Hence the proof again follows from these facts, together with the 
 Pressley/Segal Birkhoff factorization and the topology of 
 $\widehat \hh_*^-$ and $\widehat \hh^+$ given by Lemma \ref{topologylemma}.
\end{proof}
 
 Note that in some cases, such as when $\hat \theta$ is a $\C$-linear involution of the 
first kind which restricts to an inner automorphism of $G^\C$, one has an isomorphism
$\Psi: \Lambda G^\C \to \Lambda G^\C_{\hat \theta}$.
However, the Birkhoff splitting obtained from this isomorphism is not the same as
the one just given, because $\Psi (\Lambda^\pm G^C) \neq \Lambda^\pm G^\C_{\hat \theta}$. 
For practical applications, such as dressing and the DPW method,
 we are really interested in decompositions of the form
$\Lambda^- G^\C_{\hat \theta} \, \cdot \, \Lambda^+ G^\C_{\hat \theta}$, as described here.

%****************:

%============================================
%============================================
\section{Proof of Theorem \ref{thm2}}
Generalizations of the standard Iwasawa decomposition of $\Lambda G^\C$, namely (\ref{iwasawa}), to subgroups $\kk$ of $\Lambda G^\C$ and arbitrary involutions
 $\hat \tau$  of the second kind, were considered in \cite{branderdorf}.
 That is, decompositions of the form:
\beq \label{generaliwasawa}
\kk = \kk_{\hat \tau} \cdot \kk^+,
\eeq
 where  $\kk_{\hat \tau}$ is the fixed point subgroup,
 $\kk^+ = \kk \cap \Lambda^+ G^\C$, and where
 the $\kk_{\hat \tau}$ factor is unique up to right multiplication by 
 the constant subgroup
 $\kk_{\hat\tau}^0 :=\kk_{\hat \tau} \cap G^\C$.
 In general such a decomposition exists on a neighbourhood of the identity, but not globally.
 For example, in the standard Iwasawa decomposition (\ref{iwasawa}), if the real form
 were non-compact, then the decomposition would not be global, even on the identity 
 component of $\Lambda G^\C$.
 
Theorem \ref{thm2} states that the decomposition is global for 
$\kk = \hh :=   \Lambda G^\C_{\rhohat} $.
Note that this is not just a special case of the standard Iwasawa decomposition
(\ref{iwasawa}). Taking the case $\kk = \left[ \Lambda G^\C \right]_e$,
and
$\hat \tau$ is an antilinear involution of the second kind
for a \emph{non}-compact real form, then, as was already mentioned, the
decomposition (\ref{generaliwasawa}) is not global. See for example, the case of 
$\Lambda SU(1,1)$, studied in \cite{brs}.
But if one now restricts this to 
$\Lambda G^\C_{\rhohat} \subset \kk$, assuming that $\rhohat$
commutes with $\hat \tau$,
then the theorem holds.

We need a result from \cite{branderdorf}. Note that a Birkhoff decomposable
subgroup $\kk$ of $\Lambda G^\C$ means a connected subgroup which
has the property that $\left([\kk^-]_e \cdot [\kk^+]_e \right) \cap ([\kk^+]_e \cdot [\kk^-]_e)$
 is open and dense in $[\kk]_e$, and  this property is satisfied by our group $\hh$.
 Define the \emph{right big cell} of $\Lambda G^\C$
in analogue with the (left) big cell $\mathcal{B}$, interchanging $+$ and $-$, i.e.
$\mathcal{RB} := \Lambda^+ G^\C \, \cdot \, \Lambda ^- G^\C$. 

%**************
\begin{theorem} \label{thm4} (From Theorem 3 in \cite{branderdorf}.) 
Suppose $\kk$ is any Birkhoff decomposable subgroup of $\Lambda G^\C$,
and $\hat \tau$ is any involution of the second kind of $\Lambda G^\C$,
which restricts to an involution of $\kk$.
Set 
\bdm
\mathcal{U} := \{x \in \kk ~|~ \exists ~z_\tau \in \kk_{\hat \tau},~ y_+ \in \kk^+, ~
  \textup{s.t.} ~ x = z_\tau y_+ \}.
\edm
\begin{enumerate}
\item
If the constant subgroup $\kk^0$ 
is connected and compact, then every element $x \in \kk$ which has
the property that $x^{-1} \hat \tau x$ is in the 
identity component of $\kk \cap \mathcal{RB}$ is also an element of $\mathcal{U}$.
\item If $x = z_\tau y_+ \in \mathcal{U}$, then the factor $z_\tau$ is unique
up to right multiplication by an element of $\kk^0$. 
\item The map $\mathcal{U} \to \frac{\kk_{\hat \tau}}{\kk^0}$ given by
$x = z_\tau y_+ \mapsto [z_\tau]$ is real analytic.
\end{enumerate}
\end{theorem}
%************

\noindent \textbf{Proof of Theorem \ref{thm2}:} In Theorem \ref{thm4}, for 
the case $\kk = \hh$, we want to show that $\mathcal{U} = \hh$.
By  Theorem \ref{thm1}, and Remark \ref{remark1}, the right big cell is the whole group, i.e. 
$\hh \cap \mathcal{RB} = \hh$. Note that $\hh^0 = G$, which is assumed connected.
Hence the conditions of Theorem 
\ref{thm4} for $x$ to be an element of $\mathcal{U}$ are met for any $x \in \hh$.
\qed \\

\subsection{Generalization of Theorem \ref{thm2}}
The following generalization has the identical proof of Theorem \ref{thm2}, using
Theorem \ref{thm1a}:
\begin{theorem} \label{thm2a}
Let 
 $\widehat{\hh} = \Lambda G^\C_{\rhohat \hat \theta_1 ... \hat \theta_k}$, as defined
 in Theorem \ref{thm1a}. Suppose that $\hat \tau$ is a $\C$-linear or $\C$-antilinear
  involution of
 the second kind which commutes with all of $\rhohat$ and $\hat \theta_i$.
 If the constant subgroup $\widehat{\hh}^0 = G_{\theta_1 ... \theta_k}$ is connected
 then Theorem \ref{thm2} is also valid for $\widehat{\hh}$.
\end{theorem}

%==============================================
\section{An application of the theorems} \label{iwasawaapp}
As mentioned in Section \ref{applications}, Theorems \ref{thm1} and \ref{thm2} are relevant
to the application of the DPW method to isometric immersions of space forms (and 
generalizations). Isometric immersions of the hyperbolic space ${\mathbb H}^n$
into ${\mathbb E}^{2n-1}$ - equivalent to the case $k=n-1$ in the discussion below -
were shown to be an integrable system in work of Terng and Tenenblat in \cite{terng1980, terngtenenblat}, where B\"acklund transformations and soliton solutions are defined. 
The loop group formulation used here is due to Ferus and Pedit \cite{feruspedit1996},
where a generalized AKS theory was developed to produce ``finite type" solutions.

It was shown in \cite{branderdorf} that the loop group maps
for  isometric immersions with flat normal bundle 
of space forms locally correspond, in a one to one manner,
with curved flats
via the Birkhoff and generalized Iwasawa splittings. These have  been 
studied further in \cite{brander2}.

\subsection{Specific case}
 An interesting problem is the case of isometric 
immersions with flat normal bundle of the hyperbolic space ${\mathbb H}^n$ into 
either a hyperbolic space ${\mathbb H}_c^{n+k}$, where $-1<c$, or into a sphere
${\mathbb S}^{n+k}$.  Solutions to both of these problems are obtained from just
one map into a loop group, as follows (for details not proved here, see \cite{brander2}):
let $G^\C = SO(n+k+1,\C)$, where $k \geq n-1$,  and define the following three involutions, 
$\hat \sigma$, $\hat \rho$ and $\hat \tau$
on $\Lambda G^\C$ by the formulae:
\bdm
(\hat \sigma x)(\lambda) := \sigma (x(-\lambda)), \hspace{1cm}
(\hat \rho x) (\lambda) := \rho (x(-\bar \lambda)), \hspace{1cm}
(\hat \tau x) (\lambda) := \tau (x(\lambda^{-1})),
\edm
where $\rho$ is complex conjugation and 
\beqas
\sigma = \textup{Ad}_P, \hspace{2cm} P = \bbar I_n & 0 \\ 0 & -I_{k+1} \ebar,\\
\tau = \textup{Ad}_Q, \hspace{2cm} Q = \bbar I_{n+1} & 0 \\ 0 & -I_{k} \ebar,
\eeqas
and $I_r$ denotes an $r \times r$ identity matrix. 
Note that all three involutions commute. Set
\bdm
\hh = \Lambda G^\C_{\hat \rho \hat \sigma},
\edm
the fixed point subgroup 
with respect to $\hat \rho$ and $\hat \sigma$.  
Now $\hat \rho$ is an antilinear involution of the first kind, 
$G^\C_\rho = SO(n+k+1,\real) =:G$ is
compact, and $\hat \sigma$ is also an involution of the first kind. Thus we can apply 
 Theorem \ref{thm1a} to $\hh$. Moreover, 
$\hh^0 = G_\sigma = SO(n) \times SO(k+1)$, which is connected, so we can
also  apply Theorem \ref{thm2a} to $\hh$, with respect to the involution of 
the second kind $\hat \tau$.

 Since  all three involutions commute, 
$\hat \tau$ restricts to an involution of $\hh$. Let $\hh_{\hat \tau}$ denote the 
fixed point subgroup.

Consider the set of  $\mathcal{C}^\infty$ immersions  $f: \real^n \to \hh_{\hat \tau}/\hh_{\hat \tau}^0$,
such that, for any frame $F: \real^n \to \hh_{\hat \tau}$ for $f$, the Maurer-Cartan
form, $F^{-1}\dd F$, has a  Fourier expansion in $\lambda$ which is polynomial
of the form $\alpha_{-1} \lambda^{-1} + \alpha_0 + \alpha_1 \lambda$, 
where $\alpha_i$ are $\mathfrak{g}^\C$-valued 1-forms. This condition does not
depend on the choice of frame. Denote the set of such immersions by
\bdm
\mathcal{S} := \{ f: \real^n \to \frac{\hh_{\hat \tau}}{\hh_{\hat \tau}^0} ~|~ 
 f \textup{ regular}, ~F^{-1} \dd F = \sum_{i=-1}^1 \alpha_i \lambda^i, ~~\textup{for any frame }F \}.
\edm
If $F$ is a frame for $f \in \mathcal{S}$, then $F$ can be extended 
holomorphically in $\lambda$ to $\C^*$.
For a fixed pure imaginary value of $\lambda$, $\lambda_0 \in i \real$, the map
$F_{\lambda_0}: \real^n \to G^\C$ takes values in the real form $SO(n+k+1)$ 
(c.f. Lemma \ref{branderdorflemma}).  One can also verify  that
for a fixed value of $\lambda$ on the \emph{unit circle}, 
$\lambda_0 \in {\mathbb S}^1$, $F_{\lambda_0}$ takes values in a group 
isomorphic to the non-compact real form $SO(n+k,1)$. 

 Let $\tilde{f}_\lambda$ denote the $(n+1)$'st column of $F_\lambda$.
Multiplication on the right by 
$\hh_{\hat \tau}^0 = SO(n) \times \{I\} \times SO(k)$ leaves this column fixed,
so $\tilde{f}_\lambda$ is well defined by $f \in \mathcal{S}$.
Assume that $\tilde{f}_\lambda$ is an immersion (a generic condition in these dimensions).
Then:
\begin{enumerate}
\item
 For $\lambda_0 \in i \real$, the map
 $\tilde{f}_{\lambda_0}: \real^n \to {\mathbb S}^{n+k}$, with the induced 
metric, has flat normal bundle and 
constant curvature $c_{\lambda_0} \in (-\infty,0)$.
\item
 For 
$\lambda_0 \in {\mathbb S}^1$, the map $\tilde{f}_{\lambda_0}: \real^n \to {\mathbb H}^{n+k}$
has flat normal bundle and constant curvature $c_{\lambda_0} \in (-\infty,-1)$.
\end{enumerate} 

Conversely, any immersion of the two types just described is associated  to an element
of $\mathcal{S}$ whose $(n+1)$'st column is an immersion.

\subsubsection{The dressing action on $\mathcal{S}$}
We describe how to use the Iwasawa-type decomposition, Theorem \ref{thm2a},
to obtain a group action by $\hh^-$ on $\mathcal{S}$. Given  $g_- \in \hh^-$ and
a frame $F$ for $f \in \mathcal{S}$, define a new frame
 $\hat F$ by the (pointwise at $x \in \real^n$) Iwasawa
decomposition
\bdm
g_- \, F(x) = \hat F(x) \, g_+(x), \hspace{1.5cm} \hat F(x) \in \hh_{\hat \tau}, \hspace{1cm}
   g_+(x) \in \hh^+.
\edm 
The equivalence class $[\hat F(x)] \in \hh_{\hat \tau}/ \hh_{\hat \tau}^0$ is 
well defined and depends smoothly on $g_- F(x)$
, so we can choose $\hat F$ smoothly on the 
 domain of $F$.  Since $\hh_{\hat \tau}^0$ consists of constant loops, it is also
 clear that $[\hat F]$ is well defined, independent of the choice of frame $F$
 for $[F]$.
For a smooth choice of $\hat F$,  $g_+$ is also smooth, and we can compute the
Maurer-Cartan form of $\hat F$:
\bdm
\hat F^{-1} \dd \hat F = g_+^{-1} F^{-1} \dd F g_+ + g_+ ^{-1} \dd g_+.
\edm
We have $F^{-1} \dd F = \sum_{i=-1}^1 \alpha_i \lambda^i$ and we can expand
$g_+ = a_0 + a_1 \lambda + ....$. Hence
\bdm
\hat F^{-1} \dd \hat F = a_0^{-1} \alpha_{-1} a_0 \, \lambda^{-1} + ...
\edm
Since $\hat F$ is fixed by $\hat \tau$, which takes $\lambda \mapsto \lambda^{-1}$,
it follows that 
\bdm
\hat F^{-1} \dd \hat F = \hat\alpha_{-1} \lambda^{-1} + \hat\alpha_0
  + \hat\alpha_1 \lambda.
\edm
The condition that $[\hat F]$ is an immersion into
 $\hh_{\hat\tau}/\hh_{\hat\tau}^0$ can be read
off the 1-form $\hat \alpha_{-}$, and one can deduce that $[\hat F]$ is immersed if
$[F]$ is immersed, which is to say $[\hat F] \in \mathcal{S}$.  It is straightforward
to check that this defines a group action by $\hh^-$ on $\mathcal{S}$.

One cannot guarantee that the projection to the $(n+1)$'st column of the 
dressed solution will be  everywhere regular.  This is true with
respect to all the loop group methods for this particular problem (including
the generalized AKS theory of \cite{feruspedit1996}), and is
 not related to
the methods themselves, which globally produce immersions in the set $\mathcal{S}$.
Thus from the loop group point of view, the natural object to consider here is
$\mathcal{S}$, whose elements correspond to constant curvature submanifolds which
may be non-immersed at some points, but which lift to  immersions into either
$\frac{SO(n+k+1)}{SO(n) \times\{1\} \times SO(k)}$ or 
$\frac{SO(n+k,1)}{SO(n) \times \{1\} \times SO(k)}$ respectively for 
cases (1) and (2) above.

\subsubsection{The DPW method for $\mathcal{S}$}
The DPW method, described in \cite{branderdorf}, 
  uses the ``non-global" versions
of Theorems \ref{thm1a} and \ref{thm2a} to establish,
after fixing a basepoint in $\real^n$,
a (local) bijection between elements of $\mathcal{S}$ and regular curved flats in
$SO(n+k+1)/(SO(n) \times SO(k+1))$.
This bijection can now be seen to be global.

  These conclusions  are of interest because it is an open problem
whether or not \emph{complete} solutions exist for the immersions 
of types (1) or (2) above. This is 
equivalent to the problem of the existence of complete
isometric immersions with flat normal bundle 
of the hyperbolic space ${\mathbb H}^n$ into the Euclidean space ${\mathbb E}^{n+k}$,
a natural generalization of Hilbert's non-immersibility theorem of the hyperbolic
plane \cite{hilbert1}. 
In the higher dimensional generalization, non-immersibility has been proven for a 
\emph{non}-simply-connected hyperbolic space in a series of papers \cite{nikolayevsky, pedit, xavier}. See the
discussion in \cite{rrgh} for other possible generalizations to which these
comments also apply.
 We can conclude that any singularities arising
in the immersions constructed are not related to the failure of a loop group
decomposition.  This is in contrast to the case of constant mean curvature
surfaces in Minkowski 3-space, which are studied in \cite{brs}, where \emph{all}
singularities (and there are many), other than branch points, arise as a result of the lack of the
global property for the Iwasawa decomposition used.

\providecommand{\bysame}{\leavevmode\hbox to3em{\hrulefill}\thinspace}
\providecommand{\MR}{\relax\ifhmode\unskip\space\fi MR }
\providecommand{\MRhref}[2]{%
  \href{http://www.ams.org/mathscinet-getitem?mr=#1}{#2}
}
\providecommand{\href}[2]{#2}

\end{document}